\documentclass[12pt,reqno]{amsart}

\usepackage[margin=1in]{geometry}
\usepackage[T1]{fontenc}
\usepackage[utf8]{inputenc}
\usepackage{amsmath,amsfonts,amssymb,amsthm,amscd,mathtools}
\usepackage{hhline}
\usepackage{float}
\usepackage{graphicx}
\usepackage[titletoc,title]{appendix}
\usepackage[dvipsnames]{xcolor}
\usepackage{caption}
\usepackage{listings}
\usepackage{hyperref}

\usepackage[table,dvipsnames]{xcolor}
\usepackage{colortbl}

\numberwithin{equation}{section}

\newtheorem{thm}{Theorem}[section]

\newtheorem{lem}[thm]{Lemma}
\newtheorem{prop}[thm]{Proposition}

\newtheorem{defi}[thm]{Definition}

\definecolor{deepblue}{rgb}{0,0,0.5}
\definecolor{deepred}{rgb}{0.6,0,0}
\definecolor{deepgreen}{rgb}{0,0.5,0}
\definecolor{ao}{rgb}{0.0,0.5,0.0}

\newcommand{\bburl}[1]{\textcolor{blue}{\url{#1}}}
\newcommand{\seqnum}[1]{\href{https://oeis.org/#1}{\rm\underline{#1}}}

\begin{document}

\title{Counting Schreier Sets Under Neighborhood Conditions}

\author[B. Chen]{Ben Chen}
\email{\textcolor{blue}{\href{mailto:bec006@ucsd.edu}{bec006@ucsd.edu}}}
\address{Department of Mathematics, UC San Diego, La Jolla, CA 92093, USA}

\author[H. V. Chu]{H\`ung Vi\d{\^e}t Chu}
\email{\textcolor{blue}{\href{mailto:hchu@wlu.edu}{hchu@wlu.edu}}}
\address{Department of Mathematics, Washington and Lee University, Lexington, VA 24450, USA}  

\author[C. F. Foss]{Corbin F. Foss}
\email{\textcolor{blue}{\href{mailto:corbinwatts09@gmail.com}{corbinwatts09@gmail.com}}}
\address{Department of Mathematics, University of Akron, Akron, OH 44325, USA}

\author[J. Luo]{Jie Luo}
\email{\textcolor{blue}{\href{mailto:jieluo589@gmail.com}{jieluo589@gmail.com}}}
\address{School of Mathematical Sciences, University of Nottingham, Nottingham, NG7 2RD, United Kingdom}

\author[S. J. Miller]{Steven J. Miller} 
\email{\textcolor{blue}{\href{mailto:sjm1@williams.edu}{sjm1@williams.edu}, \href{mailto:Steven.Miller.MC.96@aya.yale.edu}{Steven.Miller.MC.96@aya.yale.edu}}}
\address{Department of Mathematics, Williams College, MA 01267, USA}

\author[R. Ren]{Richard Ren}
\email{\textcolor{blue}{\href{mailto:rren@uchicago.edu}{rren@uchicago.edu}}}
\address{Department of Mathematics, University of Chicago, Chicago, IL 60637, USA}

\author[G. Tresch]{Garrett Tresch}
\email{\textcolor{blue}{\href{mailto:treschgt@ucmail.uc.edu}{treschgt@ucmail.uc.edu}}}
\address{Department of Mathematics, University of Cincinnati Blue Ash, Blue Ash, OH  45236, USA}

\thanks{This work was partially supported by the National Science Foundation DMS2341670. We thank the
participants at Polymath Jr. 2026 for helpful discussions.}

\begin{abstract}
We count Schreier sets that satisfy a neighborhood condition, including $k$-clustered, $k$-consecutive-free, $k$-neighbored, $k$-isolated, and closed under integral $2$-averages. For
the first four conditions, we determine the initial counts and prove linear recurrence relations.
For the last condition, we prove a recurrence that involves the divisor counting function.
\end{abstract}

\maketitle

\tableofcontents

\section{Introduction}

A finite nonempty set $F\subset \mathbb N$ is called \textit{Schreier} if $\min F\ge |F|$. Let us use $[n]$ to denote the set $\{1, 2, \ldots, n\}$ and use $[a,b]$ to denote the set of all integers between $a$ and $b$, inclusively. Bird discovered an unexpected connection between Schreier sets and the Fibonacci numbers $(F_n)_{n=1}^\infty$, given by $F_1 = F_2 = 1$ and $F_n = F_{n-1} + F_{n-2}$ for $n\ge 3$ \cite{bird2012shcreierfib}. Bird defined 
\begin{equation}\label{e30}\mathcal{S}_{n}\ :=\ \{F\subset[n]\,:\, n\in F\mbox{ and }  F\mbox{ is Schreier}\}\end{equation}
and showed that 
\begin{equation}\label{bfs}|\mathcal{S}_n|\ =\ F_n\mbox{ for all }n\ge 1.\end{equation} Various generalizations of \eqref{e30} have accummulated over the years. Beanland et al.\ studied the condition $q\min F\ge p|F|$ with $(p,q)\in \mathbb{N}^2$ and witnessed recurrences whose depth is equal to $p+q$ \cite{BCF}. More recently, Beanland et al.\ counted unions of Schreier sets and described the recurrences recursively using characteristic polynomials \cite{BGHH}.  By counting multisets, Chu et al.\ connected the Schreier condition to the $s$-step Fibonacci numbers and derived new recurrence families and alternating-sign Pascal-type recurrences \cite{CGKMTV, CIMSZ}. Last but not least, interested readers may refer to the exposition \cite{C26} for common proof techniques in the area together with several new contributions. 

We introduce several conditions that build on the usual Schreier condition, all centered on restricting the neighborhood of each element in a set. The guiding principle is that linear recurrences can often be uncovered experimentally, through patterns suggested by initial data, and then explained combinatorially. We partition the underlying families of objects in a natural way and construct bijections that translate these partitions into recurrence relations. Often, these partitions involve splitting families of Schreier sets into those where the Schreier condition is achieved as an equality and where the inequality is strict. Specifically, we shall say a set $F\subset \mathbb{N}$ is \textit{maximal Schreier} if $\min F=|F|$ and \textit{nonmaximal Schreier} if $\min F> |F|$.

In Section~\ref{sect_k_clustered}, we study Schreier sets that are $k$\textit{-clustered}, meaning that every element is within a cluster of $k$ consecutive integers. Precisely, for each $k\in\mathbb N$, a nonempty set $F\subset\mathbb N$ is called \emph{$k$-clustered} if for every $a\in F$, there exists an interval of consecutive integers $I\subseteq F$ such that $a\in I$ and $|I|\ge k$. We are interested in the collection
\begin{equation}
\mathcal{C}_{k,n}\ :=\ \{F\subset [n]\,: \, n\in F, F\text{ is Schreier, and }F\text{ is }k\text{-clustered}\}.
\end{equation}
Table \ref{tab:k-clustered-data} computes $|\mathcal{C}_{k,n}|$ with small $k$ and $n$.

\begin{table}[H]
\centering
\begin{tabular}{|c|c|c|c|c|c|c|c|c|c|c|c|c|c|c|c|c|c|c|c|}
\hline
$n$& $1$ & $2$ & $3$ & $4$ & $5$ & $6$ & $7$ & $8$ & $9$ & $10$ & $11$
& $12$ & $13$ & $14$ & $15$ & $16$ & $17$ & $18$ & $19$ \\
\hline
$|\mathcal C_{2,n}|$
& $0$ & $0$ & $1$ & $1$ & $2$ & $2$ & $3$ & $4$ & $6$ & $9$ & $13$
& $19$ & $27$ & $39$ & $56$ & $81$ & $117$ & $169$ & $244$  \\
\hline
$|\mathcal C_{3,n}|$
& $0$ & $0$ & $0$ & $0$ & $1$ & $1$ & $2$ & $2$ & $3$ & $3$ & $4$
& $5$ & $7$ & $10$ & $14$ & $20$ & $27$ & $37$ & $49$  \\
\hline
$|\mathcal C_{4,n}|$
& $0$ & $0$ & $0$ & $0$ & $0$ & $0$ & $1$ & $1$ & $2$ & $2$ & $3$
& $3$ & $4$ & $4$ & $5$ & $6$ & $8$ & $11$ & $15$ \\
\hline
$|\mathcal C_{5,n}|$
& $0$ & $0$ & $0$ & $0$ & $0$ & $0$ & $0$ & $0$ & $1$ & $1$ & $2$
& $2$ & $3$ & $3$ & $4$ & $4$ & $5$ & $5$ & $6$ \\
\hline
\end{tabular}
\caption{The first few values of $(|\mathcal C_{k,n}|)_{n=1}^\infty$ with $2\le k\le 5$.}
\label{tab:k-clustered-data}
\end{table}

\begin{thm}\label{thm:k-clustered-recurrence}
Let $k\ge 2$. We have 
\begin{equation}
|\mathcal C_{k,n}|\ =\ 
\begin{cases}
    0, & \text{if } 1\le n\le 2k-2,\\
    1, & \text{if } 2k-1\le n\le 2k,\\
    2, & \text{if } n = 2k+1.
\end{cases}
\end{equation}
For $n\ge 2k+2$, we have
\begin{equation}
|\mathcal C_{k,n}|\ =\ |\mathcal C_{k,n-1}|+|\mathcal C_{k,n-2}|-|\mathcal C_{k,n-3}|+|\mathcal C_{k,n-(2k+1)}|.
\end{equation}
\end{thm}

Section~\ref{sect_k-free} heads in the opposite direction of our first result, considering Schreier sets that are $k$\textit{-consecutive free}, meaning there are no blocks of $k$ consecutive integers. 
    For each $k\ge 2$ and $n\in \mathbb{N}$, we are interested in the collection 
    \begin{equation}\mathcal{E}_{k,n}\ :=\ \{F\subset[n]\,:\,n\in F, F\mbox{ is Schreier},\text{ and }[m,\,m\,+\,k\,-\,1]\not\subset F,\,\text{for all }m\in F\}.\end{equation}
    A python code generates Table \ref{tab:k-consecutive}.
    \begin{table}[H]
    \centering
    \begin{tabular}{|c|c|c|c|c|c|c|c|c|c|c|c|c|c|c|c|c|c|c|}
    \hline
        $n$                 & $1$ & $2$ & $3$ & $4$ & $5$ & $6$ & $7$  & $8$  & $9$  & $10$ & $11$ & $12$  & $13$  & $14$  & $15$  & $16$  & $17$   & $18$   \\
        \hline
        $|\mathcal{E}_{2,n}|$ & $1$ & $1$ & $1$ & $2$ & $3$ & $4$ & $6$  & $9$  & $13$ & $19$ & $28$ & $41$  & $60$  & $88$  & $129$ & $189$ & $277$  & $406$  \\
        \hline
        $|\mathcal{E}_{3,n}|$ & $1$ & $1$ & $2$ & $3$ & $4$ & $7$ & $11$ & $17$ & $27$ & $42$ & $66$ & $104$ & $163$ & $256$ & $402$ & $631$ & $991$  & $1556$ \\
        \hline
        $|\mathcal{E}_{4,n}|$ & $1$ & $1$ & $2$ & $3$ & $5$ & $8$ & $12$ & $20$ & $32$ & $51$ & $82$ & $131$ & $210$ & $336$ & $538$ & $862$ & $1380$ & $2210$ \\
        \hline
        $|\mathcal{E}_{5,n}|$ & $1$ & $1$ & $2$ & $3$ & $5$ & $8$ & $13$ & $21$ & $33$ & $54$ & $87$ & $140$ & $226$ & $364$ & $587$ & $946$ & $1525$ & $2458$ \\
        \hline
    \end{tabular}
    \caption{The first few terms of $(|\mathcal{E}_{k,n}|)_{n=1}^\infty$ with $2\le k \le 5$. The sequence $(|\mathcal{E}_{2,n}|)$ is \seqnum{A000930} in the On-Line Encyclopedia of Integer Sequences (OEIS) \cite{OEIS}, which satisfies $|\mathcal{E}_{2,n}| = |\mathcal{E}_{2, n-1}| + |\mathcal{E}_{2,n-3}|$. Similarly, the sequences $(|\mathcal{E}_{3,n}|)$ and $(|\mathcal{E}_{4,n}|)$ are \seqnum{A222122} and \seqnum{A117761}, respectively. The sequence $(|\mathcal{E}_{5,n}|)$ is not yet available on the OEIS.}
    \label{tab:k-consecutive}
    \end{table}

    \begin{thm}\label{thm:k-consecutive-free-recurrence}
    Let $k\ge 2$. We have
    \begin{equation}|\mathcal{E}_{k,n}|\ =\ F_n \mbox{ if } 1\le n\le 2k-2,\mbox{ and } |\mathcal{E}_{k,2k-1}|\ =\ F_{2k-1}-1.\end{equation}
    For 
    $n\ge 2k$, we have
    \begin{equation}|\mathcal{E}_{k,n}|\ =\ |\mathcal{E}_{k,n-1}| + |\mathcal{E}_{k,n-3}|+\cdots +|\mathcal{E}_{k,n-(2k-1)}|\ =\ \sum_{i=1}^{k}|\mathcal{E}_{k,n-(2i-1)}|.\end{equation}
    \end{thm}

In Section \ref{sect_k_neigh}, we study Schreier sets whose every element has a nearby neighbor. For $k\in\mathbb N$, a set $F\subset \mathbb N$ is called \emph{$k$-neighbored} if for every $a\in F$, there exists $b\in F\backslash\{a\}$ such that $|a-b|\le k$. Equivalently, $F$ is $k$-neighbored if
\begin{equation}
([a-k,a+k]\cap F)\backslash\{a\}\ \neq\ \emptyset, \mbox{ for every }a\in F.
\end{equation}
For $k,n\in\mathbb N$, we define
\begin{equation}
\mathcal D_{k,n}\ :=\ \{F\subset [n]: n\in F, F\text{ is Schreier, and }F\text{ is }k\text{-neighbored}\}.
\end{equation}

Table \ref{tab:k-neighbored-data} collects data on the size of $\mathcal{D}_{k,n}$.

\begin{table}[H]
\centering
\begin{tabular}{|c|c|c|c|c|c|c|c|c|c|c|c|c|c|c|c|c|c|c|}
\hline
$n$
& $1$ & $2$ & $3$ & $4$ & $5$ & $6$ & $7$ & $8$ & $9$ & $10$ & $11$ & $12$ & $13$ & $14$ & $15$ & $16$ & $17$ & $18$ \\
\hline
$|\mathcal D_{1,n}|$
& $0$ & $0$ & $1$ & $1$ & $2$ & $2$ & $3$ & $4$ & $6$ & $9$ & $13$ & $19$ & $27$ & $39$ & $56$ & $81$ & $117$ & $169$ \\
\hline
$|\mathcal D_{2,n}|$
& $0$ & $0$ & $1$ & $2$ & $3$ & $5$ & $7$ & $10$ & $15$ & $23$ & $36$ & $57$ & $90$ & $141$ & $220$ & $342$ & $531$ & $825$ \\
\hline
$|\mathcal D_{3,n}|$
& $0$ & $0$ & $1$ & $2$ & $4$ & $6$ & $10$ & $15$ & $23$ & $35$ & $55$ & $87$ & $139$ & $223$ & $358$ & $574$ & $918$ & $1466$ \\
\hline
$|\mathcal D_{4,n}|$
& $0$ & $0$ & $1$ & $2$ & $4$ & $7$ & $11$ & $18$ & $28$ & $44$ & $69$ & $109$ & $174$ & $279$ & $449$ & $724$ & $1168$ & $1884$ \\
\hline
$|\mathcal D_{5,n}|$
& $0$ & $0$ & $1$ & $2$ & $4$ & $7$ & $12$ & $19$ & $31$ & $49$ & $78$ & $124$ & $198$ & $317$ & $510$ & $822$ & $1327$ & $2144$ \\
\hline
$|\mathcal D_{6,n}|$
& $0$ & $0$ & $1$ & $2$ & $4$ & $7$ & $12$ & $20$ & $32$ & $52$ & $83$ & $133$ & $213$ & $342$ & $550$ & $886$ & $1430$ & $2310$\\
\hline
\end{tabular}
\caption{The first few values of $(|\mathcal D_{k,n}|)_{n=1}^\infty$ with $1\le k\le 6$.}
\label{tab:k-neighbored-data}
\end{table}

\begin{thm}\label{thm:k-neighbored-recurrence}
Let $k\ge 1$. For $n\le k+2$, we have
\begin{equation}|\mathcal{D}_{k,n}|\ =\ F_n - 1.\end{equation}
For $n\ge k+3$, we have
\begin{equation}
|\mathcal D_{k,n}|
\ =\
|\mathcal D_{k,n-1}|
+
|\mathcal D_{k,n-2}|
-
|\mathcal D_{k,n-k-2}|
+
\sum_{j=k+4}^{2k+3}|\mathcal D_{k,n-j}|,
\end{equation}
with the convention that $|\mathcal{D}_{k, i}| = 0$ if $i \le 0$.
\end{thm}

Next, Section~\ref{sect_k_isolated} establishes a recurrence for $k$\textit{-isolated} Schreier sets which are the opposite of $k$-neighbored sets. For each $k\in \mathbb{N}$, a  nonempty set is called \textit{$k$-isolated} if for every $a\in D$, 
\begin{equation}([a-k, a-1]\cup [a+1, a+k])\cap D \ = \ \emptyset.\end{equation}
For $k,n\in\mathbb{N}$, we are interested in the collection 
\begin{equation}\mathcal{I}_{k, n}\ :=\ \{F\subset[n]\,:\, n\in F, F\mbox{ is Schrerier}, \mbox{ and }F\mbox{ is }k\mbox{-isolated}\}.\end{equation}
Table \ref{tab_iso} records $(|\mathcal{I}_{k,n}|)_{n=1}^{20}$ with $k\in \{1, 2, 3, 4\}$.
\begin{table}[H]
    \centering
    \begin{tabular}{|c|*{20}{c|}}
        \hline
        $n$                 
        & $1$ & $2$ & $3$ & $4$ & $5$ & $6$ & $7$ & $8$ & $9$ & $10$
        & $11$ & $12$ & $13$ & $14$ & $15$ & $16$ & $17$ & $18$ & $19$ & $20$ \\
        \hline
        $|\mathcal{I}_{1,n}|$
        & $1$ & $1$ & $1$ & $2$ & $3$ & $4$ & $6$ & $9$ & $13$ & $19$
        & $28$ & $41$ & $60$ & $88$ & $129$ & $189$ & $277$ & $406$ & $595$ & $872$ \\
        \hline 
        $|\mathcal{I}_{2,n}|$
        & $1$ & $1$ & $1$ & $1$ & $2$ & $3$ & $4$ & $5$ & $7$ & $10$
        & $14$ & $19$ & $26$ & $36$ & $50$ & $69$ & $95$ & $131$ & $181$ & $250$ \\
        \hline 
        $|\mathcal{I}_{3,n}|$
        & $1$ & $1$ & $1$ & $1$ & $1$ & $2$ & $3$ & $4$ & $5$ & $6$
        & $8$ & $11$ & $15$ & $20$ & $26$ & $34$ & $45$ & $60$ & $80$ & $106$ \\
        \hline
        $|\mathcal{I}_{4,n}|$
        & $1$ & $1$ & $1$ & $1$ & $1$ & $1$ & $2$ & $3$ & $4$ & $5$
        & $6$ & $7$ & $9$ & $12$ & $16$ & $21$ & $27$ & $34$ & $43$ & $55$ \\
        \hline
    \end{tabular}
    \caption{The first few values of $|\mathcal{I}_{k,n}|$ with $1\le k \le 4$.}
    \label{tab_iso}
\end{table}

By definition, the collections $\mathcal{I}_{1,n}$ and $\mathcal{E}_{2,n}$ are the same, so Theorem \ref{thm:k-consecutive-free-recurrence} gives
\begin{equation}|\mathcal{I}_{1,n}|\ =\ |\mathcal{I}_{1,n-1}|+|\mathcal{I}_{1,n-3}|, \mbox{ for all }n\ge 4.\end{equation}
For larger $k$, the sequence $(|\mathcal{I}_{k,n}|)_{n=1}^\infty$ satisfies a recurrence of the same form, which gives the Skipponacci numbers in \cite[Definition 1.5]{DDKMV}.

\begin{thm}\label{thm:k-isolated-recurrence}
Let  $k\ge 1$. For $1\le n\le k+2$, we have 
$|\mathcal{I}_{k,n}| = 1$.
For $n\ge k+3$, we have
\begin{equation}
|\mathcal{I}_{k,n}|\ =\ |\mathcal{I}_{k,n-1}|+|\mathcal{I}_{k,n-k-2}|.
\end{equation}
\end{thm}

Lastly, in Section \ref{sect_2_averages}, we obtain a nonlinear recurrence from counting Schreier sets that are closed under the average of every two elements whenever the average is an integer. For $n \ge 0$, define
\begin{align}
\mathcal{G}_{2n+1}&\ :=\ \left\{ F\subset [2n+1]\,:\, \begin{matrix}2n+1\in F, F\mbox{ is Schreier, and}\nonumber\\
\mbox{for all } a,b \in F, \text{ if } (a+b)/2 \in \mathbb{Z}, \text{ then } (a+b)/2 \in F\end{matrix}\right\}.\\
\end{align}

\begin{table}[H]
\begin{tabular}{|c|c|c|c|c|c|c|c|c|c|c|c|c|c|}
\hline
$n$                      & 0 & 1 & 2 & 3 & 4 & 5  & 6  & 7  & 8  & 9  & 10 & 11 & 12\\
\hline
$|\mathcal{G}_{2n+1}|$ & 1 & 2 & 4 & 6 & 9 & 11 & 15 & 17 & 21 & 24 & 28 & 30 & 36\\
\hline
\end{tabular}
\caption{The first $13$ terms of $(|\mathcal{G}_{2n+1}|)_{n=0}^\infty$.}
\label{tab2-averages}
\end{table}

Recall that for $n\in\mathbb{N}$, the function $\sigma_0(n)$ counts the number of positive divisors of $n$. Table \ref{tab2-averages} suggests a nonlinear recurrence.

\begin{thm}\label{mtsigma} For $n\in\mathbb{N}$, we have 
\begin{equation}|\mathcal{G}_{2n+1}|\ =\ |\mathcal{G}_{2n-1}|+\sigma_0(n).\end{equation}
\end{thm}

\section{Recurrences from counting $k$-clustered Schreier sets}\label{sect_k_clustered}

In this section, we study a clustered variant of Schreier sets. We decompose $\mathcal C_{k,n}$ into possibly overlapping subcollections depending on whether the Schreier sets are maximal and whether the element $n-k$ belongs to these sets. We then construct explicit bijections from earlier collections $\mathcal C_{k,m}$ to these subcollections. We first need to record a useful equivalent condition of  being $k$-clustered.

\begin{defi}\normalfont
Let $F\subset \mathbb N$ be finite and nonempty. A \textit{maximal block} of consecutive integers in $F$ is a set $\{u,u+1,\ldots,v\}\subseteq F$ such that $u-1\notin F$ and $v+1\notin F$.
\end{defi}

The next lemma follows immediately from definitions. 
 
\begin{lem}\label{lem:block-characterization}
A finite nonempty set $F\subset \mathbb N$ is $k$-clustered if and only if every maximal block of consecutive integers in $F$ has length at least $k$.
\end{lem}

For $n\ge 2k+2$, we decompose $\mathcal C_{k,n}$ into three subcollections:
\begin{align}
    \mathcal{C}_{k, n}^1 &\ :=\ \{ F \subset [n]\, :\, n \in F, F \mbox{ is nonmaximal Schreier, and } F \mbox{ is }k\mbox{-clustered} \},\\
    \mathcal{C}_{k, n}^2 &\ :=\ \{ F \subset [n]\, :\, n\in F, n-k \in F, F \mbox{ is Schreier, and } F \mbox{ is }k\mbox{-clustered} \}, \mbox{ and} \\
    \mathcal{C}_{k, n}^3 &\ :=\ \{ F \subset [n] \,:\, n \in F, n-k \notin F, F\mbox{ is maximal Schreier, and } F \mbox{ is }k\mbox{-clustered} \}.
\end{align}

Define 
\begin{align}
    &f_1: \mathcal{C}_{k, n-1} \rightarrow \mathcal{C}_{k, n}^1, \quad A \mapsto A + 1,\\
    &f_2: \mathcal{C}_{k, n-2} \rightarrow \mathcal{C}_{k, n}^2, \quad A \mapsto (A + 1) \cup \{ n \},\\
    &f_3: \mathcal{C}_{k, n-(2k+1)} \rightarrow \mathcal{C}_{k, n}^3 , \quad A \mapsto (A-(\min A - |A|-k)) \cup [n-k+1, n],\mbox{ and}\\
    &f_4: \mathcal{C}_{k, n-3} \rightarrow \mathcal{C}_{k, n}^1 \cap \mathcal{C}_{k, n}^2,\quad A \mapsto (A + 2) \cup \{ n \}.
\end{align}

\begin{lem}\label{lem:f1-bijection}
    The map $f_1$ is a bijection.
\end{lem}

\begin{proof}
    Pick $A\in \mathcal{C}_{k, n-1}$. If $A$ is $k$-clustered, then $f_1(A) = A+1$ is $k$-clustered. Since $\max A = n-1$, we have $\max f_1(A) = n$. Note that 
    \begin{equation}
        \min (A+1) \ =\ \min A + 1 \ >\ |A|\ =\ |A+1|,
    \end{equation}
    so $A+1$ is nonmaximal Schreier. Hence, the map $f_1$ is well-defined. 

    Injectivity is obvious from definition. To show surjectivity, we pick $B\in \mathcal{C}_{k, n}^1$ and let $A:=B-1$. Shifting by $1$ keeps the set $k$-clustered, so $A$ is $k$-clustered.  Furthermore, $\max B = n$ gives $\max A = n-1$, and 
    \begin{equation}
        \min A \ =\ \min B - 1\ \ge\ (|B|+1) - 1\ =\ |B| \ =\ |A|
    \end{equation}
    means that $A$ is Schreier. Therefore, the set $A$ is in $\mathcal{C}_{k, n-1}$, and thus, $f_1$ is surjective.
\end{proof}

\begin{lem}\label{lem:f2-bijection}
    The map $f_2$ is a bijection.
\end{lem}

\begin{proof}
    Given a $k$-clustered set $A$ in $\mathcal{C}_{k, n-2}$, the set $(A+1)\cup\{n\}$ is also $k$-clustered because $n-1$ is in $A+1$. By definition, the set $f_2(A)$ contains $n$. Furthermore, since $n-2$ is in $A$, Lemma~\ref{lem:block-characterization} gives $n-(k+1)\in A$. It follows that $f_2(A)$ contains $n-k$. Finally,
    \begin{equation}
        \min f_2(A)\ =\ \min A + 1\ \ge\ |A| + 1\ =\ |f_2(A)|.
    \end{equation}
    We have verified that $f_2(A)$ is in $\mathcal{C}_{k,n}^{2}$, so $f_2$ is well-defined. 

    Injectivity is obvious from the definition of $f_2$. Let us show surjectivity. Pick $E\in \mathcal{C}^2_{k,n}$ and let $F:= (E\backslash \{n\})-1$. Since $E$ is $k$-clustered and $k\ge 2$, Lemma~\ref{lem:block-characterization} gives $n-1\in E$, so $\max F = n-2$. Furthermore, that $[n-k, n]\subset E$ implies that $[n-k-1, n-2]\subset F$, so $F$ is $k$-clustered. Last but not least, 
    \begin{equation}
        \min F\ =\ \min E-1\ \ge\ |E| - 1\ =\ |F|,
    \end{equation}
    so $F$ is Schreier. Therefore, the set $F$ is in $\mathcal{C}_{k, n-2}$, and thus, $f_2$ is surjective.
\end{proof}

\begin{lem}\label{lem:f3-bijection}
The map $f_3$ is a bijection.
\end{lem}

\begin{proof}
    Let $A\in \mathcal{C}_{k, n-(2k+1)}$. The set $f_3(A)$ is $k$-clustered because it is the union of a shift of the $k$-clustered set $A$ with an interval of length $k$. Furthermore, the set $f_3(A)$ contains $n$ but not $n-k$ because 
    \begin{equation}
        \max (A - (\min A - |A| - k))\ =\ n-(2k+1) - \min A + |A| +k\ \le\ n - k - 1.
    \end{equation}
    Lastly, 
    \begin{equation}
        \min f_3(A)\ =\ |A| + k\ =\ |A - (\min A-|A|-k)| + |[n-k+1, n]|\ =\ |f_3(A)|,
    \end{equation}
    so $f_3(A)$ is maximal Schreier. These show that $f_3$ is well-defined. 

    To see that $f_3$ is injective, assume that $f_3(A) = f_3(B)$ for two sets $A$ and $B$ in $\mathcal{C}_{k, n-(2k+1)}$. Since $|f_3(A)| = |A| + k$ and $|f_3(B)| = |B| + k$, we have $|A| = |B|$. From the definition of $f_3$, the equality $f_3(A) = f_3(B)$ means 
    \begin{equation}\label{e10}
        A - \min A + |A| \ =\ B-\min B + |B|, \mbox{ so } A - \min A \ =\ B - \min B.
    \end{equation}
    It follows that $\max A - \min A = \max B - \min B$. Recall that $\max A = \max B = n-(2k+1)$, so $\min A = \min B$, which, together with \eqref{e10}, implies that $A = B$. 

    It remains to show that $f_3$ is surjective. Pick $E\in \mathcal{C}^3_{k,n}$. Since $E$ is $k$-clustered, $n\in E$, and $n-k\notin E$, Lemma~\ref{lem:block-characterization} implies that $[n-k+1,n]$ is a maximal block of consecutive integers in $E$. Let $G := E\backslash [n-k+1, n]$. The set $G$ is nonempty. Otherwise, $E=[n-k+1,n]$, and since $E$ is maximal Schreier, we would have $n-k+1=k$, contradicting $n\ge 2k+2$. Let $F := G + (n-2k-1)-\max G$.
    Then 
    \begin{align}
        f_3(F)&\ =\ (F-\min F + |F| + k) \cup [n-k+1, n]\\
        &\ =\ (G - \min G  + |G| + k)\cup [n-k+1, n]\\
        &\ =\ (E\backslash [n-k+1, n]-\min E + (|E| - k) + k)\cup [n-k+1, n]\ =\ E.
    \end{align}
    We now verify that $F\in \mathcal{C}_{k, n-(2k+1)}$. First, $F$ is $k$-clustered because $F$ is a shift of $G$, which is formed by removing the maximal block $[n-k+1, n]$ from $E$. Moreover,
    \begin{equation}
        \max F\ =\ \max(G+(n-2k-1)-\max G)\ =\ n-2k-1,
    \end{equation}
    and 
    \begin{align}
   \min F &\ =\ \min G + (n-2k-1)-\max G \ = \ \min E + (n-2k-1)-\max G\\
   &\ \ge\ |E| + (n-2k-1) - (n-k-1)
   \ =\ |E| - k \ =\ |F|.
    \end{align}
    Hence, the set $F$ is in $\mathcal{C}_{k, n-(2k+1)}$, and thus, $f_3$ is surjective. 
\end{proof}

\begin{lem}\label{lem:f4-bijection}
The map $f_4$ is a bijection.
\end{lem}

\begin{proof}
    Let $A\in \mathcal{C}_{k,n-3}$ and set $E:=f_4(A)=(A+2)\cup\{n\}$. Since $A$ is $k$-clustered and $\max A = n-3$, Lemma~\ref{lem:block-characterization} implies that $[n-k-2,n-3]\subset A$, so $[n-k,n-1]\subset A+2$. Hence, $n-k\in E$, and by construction, $n\in E$.

    Shifting by $2$ keeps the set $k$-clustered, and adding $n$ only extends the last maximal block of consecutive integers because $n-1\in A+2$. Hence, $E$ is $k$-clustered. 
    
    Lastly, we have
    \begin{equation}
        \min E\ =\ \min A+2\ \ge\ |A|+2\ >\ |E|.
    \end{equation}
    Therefore, $E$ is nonmaximal Schreier. We have shown that $E\in \mathcal{C}_{k,n}^{1}\cap \mathcal{C}_{k,n}^{2}$, so $f_4$ is well-defined.

    Injectivity is obvious from the definition of $f_4$. We show that $f_4$ is surjective. Pick $E\in \mathcal{C}_{k,n}^{1}\cap \mathcal{C}_{k,n}^{2}$ and let $F:= E\backslash\{n\}-2$. Since $E\in \mathcal{C}_{k,n}^{2}$, we have $n,n-k\in E$. Lemma~\ref{lem:block-characterization} implies that $[n-k+1,n]\subset E$, so $[n-k,n]\subset E$. It follows that
    $F$ is $k$-clustered with $\max F = n-3$ and $[n-k-2,n-3]\subset F$.
    Finally, since $E$ is nonmaximal Schreier, we have
        \begin{equation}\min F\ =\ \min E-2\ \ge\ (|E|+1)-2\ =\ |E|-1\ =\ |F|.\end{equation}
    Therefore, the set $F$ is Schreier and thus, is in $\mathcal{C}_{k,n-3}$. We conclude that $f_4$ is surjective. 
\end{proof}

\begin{proof}[Proof of Theorem \ref{thm:k-clustered-recurrence}]
    Let $n\in [1,2k-2]$. If $F$ is $k$-clustered and nonempty, then $|F| \ge k$. But $F\subset [\min F, n]$, so
    \begin{equation}k\ \le\ |F|\ \le\ n-\min F +1\ \le\ (2k-2)-\min F + 1.\end{equation}
    Hence, $\min F \le k-1$, which contradicts the Schreier condition that $\min F\ge |F|\ge k$. Therefore, $|\mathcal{C}_{k,n}| = 0$ for all $n\le 2k-2$.

    For $F\in \mathcal{C}_{k,2k-1}$, we have $\min F\ge |F| \ge k$. Hence, the only 
    set in $\mathcal{C}_{k, 2k-1}$ is $[k, 2k-1]$, so $|\mathcal{C}_{k, 2k-1}| = 1$. 
    
    For $F\in \mathcal{C}_{k, 2k}$, we have $\min F\ge |F| \ge k$. Hence, the only set in
    $\mathcal{C}_{k, 2k}$ is $[k+1, 2k]$, so $|\mathcal{C}_{k,2k}| = 1$.

     For $F\in \mathcal C_{k,2k+1}$, we have $[k+2,2k+1]\subset F$. Since $|F|\ge k$, we must have $\min F\ge k$. If $k\in F$, then $k+1\in F$ and $\min F = k<k+2=|F|$, which violates the Schreier condition. If $k+1\in F$, then $\min F = k+1 = |F|$. Hence, the two sets in $\mathcal{C}_{k,2k+1}$ are $[k+2, 2k+1]$ and $[k+1, 2k+1]$. 
    
    Let $n\ge 2k+2$. 
    By Lemmas \ref{lem:f1-bijection}, \ref{lem:f2-bijection}, \ref{lem:f3-bijection}, and \ref{lem:f4-bijection}, we have
    \begin{align}
    |\mathcal{C}_{k,n}|&\ =\ |\mathcal{C}^{(1)}_{k,n}| + |\mathcal{C}^{(2)}_{k,n}| + |\mathcal{C}^{(3)}_{k,n}| - |\mathcal{C}^{(1)}_{k,n}\cap \mathcal{C}^{(2)}_{k,n}|\\
    &\ =\ |\mathcal{C}_{k,n-1}| + |\mathcal{C}_{k,n-2}| + |\mathcal{C}_{k,n-(2k+1)}| - |\mathcal{C}_{k,n-3}|.
    \end{align}
\end{proof}

\section{Recurrences from counting $k$-consecutive-free Schreier sets}\label{sect_k-free}
In this section, we study Schreier sets that do not contain intervals of length $k$, with a fixed $k \in \mathbb{N}$. 
We first record two lemmas to compute the initial terms of $(|\mathcal{E}_{k,n}|)_{n=1}^\infty$.

    \begin{lem}\label{lem: initial1-consecutive-free}
    For $k\ge 2$ and $1\le n\le 2k-2$, we have $\mathcal{E}_{k,n}=\mathcal{S}_n$, and thus, $|\mathcal{E}_{k,n}| = F_n$.  
    \end{lem}

    \begin{proof}
    The inclusion $\mathcal{E}_{k,n}\subseteq \mathcal{S}_n$ is immediate from definitions. 
    Conversely, let $F\in \mathcal{S}_n$, we have
    $|F| \le  n - \min F + 1$ because $F\subset [\min F, n]$.
    Using $n \le 2k-2$ and the Schreier condition, we obtain
    \begin{equation}
    |F| \ \le\  2k-2 - \min F + 1\  \le \ 2k-2 - |F| + 1\  =\  2k - 1 - |F|.
    \end{equation}
    Thus, $|F| \le k - 1/2 < k$. Hence, $F$ contains no $k$ 
    consecutive integers, meaning $F\in \mathcal{E}_{k,n}$. Therefore,
    $\mathcal{E}_{k,n}=\mathcal{S}_n$, and so, $|\mathcal{E}_{k,n}| = |\mathcal{S}_n| = F_n$ due to \eqref{bfs}.
    \end{proof}

    \begin{lem}\label{lem: initial2-consecutive-free}
   For $k\ge 2$, we have $\mathcal{E}_{k,2k-1}=\mathcal{S}_{2k-1}\backslash\{\{k,k+1,\ldots,2k-1\}\}$, and thus $|\mathcal{E}_{k,2k-1}| = F_{2k-1}-1$.

    \end{lem}

    \begin{proof}  Due to the Schreier condition, the only set in $\mathcal{S}_{2k-1}$ that has $k$ elements is $\{k, k+1, \ldots, 2k-1\}$. Therefore, $|\mathcal{E}_{k,2k-1}| = |\mathcal{S}_{2k-1}|-1 = F_{2k-1} - 1$ thanks to \eqref{bfs}.
    \end{proof}

To prove the recurrence in Theorem \ref{thm:k-consecutive-free-recurrence}, we introduce the following definition. 
        \begin{defi}\label{def:p-initial}\normalfont
        For $p\in\mathbb{N}$, a set $F\subset \mathbb{N}$ is \textit{$p$-initial} if
        \begin{equation}\{\min F, \min F +1,\ldots,\min F+p-1\}\subset F\quad\text{ and }\quad \min F+p\notin F.\end{equation}
    \end{defi}
For $n\ge 2k$, we decompose $\mathcal{E}_{k,n}$ into $k$ subsets:
    \begin{align}
        &\mathcal{E}_{k,n}^1\ =\ \{F\subset [n]: n\in F, \min F >|F|,\text{ and $F$ is $k$-consecutive-free}\},\\
        &\mathcal{E}_{k,n}^2\ =\ \{F\subset [n]: n\in F, \min F =|F|,\text{$F$ is $k$-consecutive-free, and $F$ is 1-initial}\},\\
        &\mathcal{E}_{k,n}^3\ =\ \{F\subset [n]: n\in F, \min F =|F|,\text{$F$ is $k$-consecutive-free, and $F$ is 2-initial}\},\\
        &\quad\vdots\nonumber\\
        &\mathcal{E}_{k,n}^i\ =\ \{F\subset [n]: n\in F, \min F =|F|,\text{$F$ is $k$-consecutive-free, and $F$ is $(i-1)$-initial}\},\\
        &\quad\vdots\nonumber\\
        &\mathcal{E}_{k,n}^k\ =\ \{F\subset [n]: n\in F, \min F =|F|,\text{$F$ is $k$-consecutive-free, and $F$ is $(k-1)$-initial}\}.
    \end{align}

We prove Theorem \ref{thm:k-consecutive-free-recurrence} by constructing bijective maps from $\mathcal{E}_{k,n-2i+1}$ to $\mathcal{E}_{k,n}^i$. Define
    \begin{align}
        &f_1:\mathcal{E}_{k,n-1}\to \mathcal{E}_{k,n}^1,\quad F\mapsto F+1 \mbox{ and}\\
&f_i: \mathcal{E}_{k, n-(2i-1)}\rightarrow \mathcal{E}_{k,n}^i, \quad F\mapsto (\{i-1,  i, \ldots, 2i-3\} + |F|)\cup (F+(2i-1)), \mbox{ for }i\in [2, k].
    \end{align}
    
    \begin{lem}\label{lem:f1-consecutive-free}
        The map $f_1$ is a bijection.
    \end{lem}
    
    \begin{proof}
        Let $F\in\mathcal{E}_{k,n-1}$. Note that $n\in f_1(F)$, and
        \begin{equation}\min f_1(F)\ =\ \min F +1\ \ge\ |F|+1\ >\ |F|\ =\ |f_1(F)|.\end{equation}  In addition, since $F$ is $k$-consecutive-free, the set $f_1(F)$ is also $k$-consecutive-free. Therefore, the map $f_1$ is well-defined.
        
        Injectivity of $f_1$ is obvious from the definition. To show surjectivity, let $E\in\mathcal{E}_{k,n}^1$ and $F:=E-1$. Since $\max E = n$, we have $\max F = n-1$. Moreover, since $E$ is nonmaximal Schreier, $\min E \ge |E|+1$, so
        \begin{equation}\min F \ =\ \min E - 1\ \ge\ |E|\ =\ |F|.\end{equation}
        In addition, $F$  is $k$-consecutive-free because $E$ is $k$-consecutive-free. Therefore, $F\in \mathcal{E}_{k,n-1}$, and $f_1({F}) = E$. Hence, the map $f_1$ is surjective.
        
        We have shown that $f_1$ is a bijection.\end{proof}
        
    \begin{lem}\label{lem:fs-consecutive-free}
        The maps $f_i$, with $2\le i\le k$, are bijections.
    \end{lem}
    
    \begin{proof}
    Fix $i\in [2, k]$. 
        We show that $f_i$ is well-defined. Let $F\in\mathcal{E}_{k,n-(2i-1)}$. Then 
        \begin{equation}f_i(F)\ =\ (\{i-1,i,\ldots,2i-3\}+|F|)\cup(F+(2i-1)).\end{equation}
       Note that
        \begin{equation}
            \min (F+(2i-1)) \ \ge\ |F|+2i-1\ >\ \max(\{i-1, i, \ldots, 2i-3\}+|F|)+1.
        \end{equation}
        Hence, $|f_i(F)| = i-1+|F| = \min f_i(F)$, so $f_i(F)$ is maximal Schreier. 

Since $|F|+2i-2\notin f_i(F)$, the set $f_i(F)$ is $(i-1)$-initial. Furthermore, $f_i(F)$ is $k$-consecutive-free because 
$F+(2i-1)$ is $k$-consecutive-free, and 
\begin{equation}\min (F+(2i-1)) \ >\ \max(\{i-1,i,\ldots,2i-3\}+|F|) + 1.\end{equation}
We have verified that $f_i(F)\in\mathcal{E}_{k,n}^i$, so $f_i$ is well-defined.

        Injectivity of $f_i$ is obvious from the definition. To show surjectivity, let $E\in\mathcal{E}_{k,n}^i$. Then $E$ is maximal Schreier and $(i-1)$-initial, so $\{|E|,|E|+1,\ldots,|E|+i-2\}\subseteq E$ and $|E|+i-1\not\in E$. Let 
        \begin{equation}F\ :=\ (E\backslash \{|E|,|E|+1,\ldots,|E|+i-2\})-(2i-1).\end{equation}
        We have
        \begin{equation}\max E\ >\  |E|+i-2;\end{equation}
        otherwise, $\max E = |E| + i - 2$ implies that $E = \{|E|,|E|+1,\ldots,|E|+i-2\}$, so $|E| = i-1$. It follows that
        \begin{equation}2k \ \le\ n\ =\ \max E \ =\ 2i-3\ \le\ 2k-3,\end{equation}
        a contradiction. Hence, $n-(2i-1)\in F$. 
        Moreover, 
        \begin{align}
            \min F &\ =\ \min (E\backslash\{|E|,|E|+1,\ldots,|E|+i-2\})-(2i-1)\\
            &\ \ge\ |E|+i-(2i-1)\ =\ |E|-(i-1)\ =\ |F|.
        \end{align}
        Hence, $F$ is Schreier. Lastly, $F$ is $k$-consecutive-free because $E$ is $k$-consecutive-free. Therefore, $F\in\mathcal{E}_{k,n-(2i-1)}$, and so, $f_i(F) = E$ gives that $f_i$ is surjective.
    \end{proof}

\begin{proof}[Proof of Theorem \ref{thm:k-consecutive-free-recurrence}]
The first statement is due to Lemmas \ref{lem: initial1-consecutive-free} and \ref{lem: initial2-consecutive-free}.
For the recurrence, we use Lemmas \ref{lem:f1-consecutive-free} and \ref{lem:fs-consecutive-free} to have
    \begin{equation}
        |\mathcal{E}_{k,n}|\ =\ \sum_{i=1}^k|\mathcal{E}_{k,n}^i|\ =\ \sum_{i=1}^k |\mathcal{E}_{k,n-(2i-1)}|.
    \end{equation}
\end{proof}


\section{Recurrences from counting $k$-neighbored Schreier sets}\label{sect_k_neigh}

We now study a variant of Schreier sets in which every element is required to have a nearby neighbor. To prove the recurrence in Theorem \ref{thm:k-neighbored-recurrence}, we assume $n\ge k+3$ and partition $\mathcal D_{k,n}$ according to whether a set is maximal Schreier and whether the closest element to $n$ has a neighbor below. We then construct bijections from earlier collections to the pieces of this partition.

\subsection{A partition of $\mathcal{D}_{k,n}$ and $\mathcal{D}_{k,n-2}$}
We say that an element $a\in F$ has a \emph{neighbor below} if
\begin{equation}
[a-k,a-1]\cap F\ \neq\ \emptyset.
\end{equation}
Equivalently, $a$ has a neighbor below if there exists $b\in F$ with $a-k\le b<a$. Define
\begin{equation}
\mathcal D^0_{k,n}\ :=\ \{F\in \mathcal D_{k,n}: F\text{ is nonmaximal Schreier}\}
\end{equation}
and define for $1\le i\le k$, 
\begin{equation}\mathcal D^i_{k,n}\ :=\
\left\{
F\in\mathcal D_{k,n}\,\middle|\,
\begin{aligned}
&F\text{ is maximal Schreier}\\
&n-i\in F\\
&[n-i+1,n-1]\cap F=\emptyset\\
&n-i\text{ has a neighbor below}
\end{aligned}
\right\},\end{equation}
and for $k+1\le i\le 2k$, 
\begin{equation}\mathcal D^{i}_{k,n}\ :=\
\left\{
F\in\mathcal D_{k,n}\;\middle|\;
\begin{aligned}
&F\text{ is maximal Schreier}\\
&n-i+k\in F\\
&[n-i+k+1,n-1]\cap F=\emptyset\\
&n-i+k\text{ has no neighbor below}
\end{aligned}
\right\}.
\end{equation}

Next, we  partition $\mathcal D_{k,n-2}$ according to how far its minimum is from its size. For $1\le i\le k$, define
\begin{equation}
\mathcal M^i_{k,n-2}
\ :=\
\{F\in \mathcal D_{k,n-2}\,:\, \min F=|F|+i-1\}
\end{equation}
and let
\begin{equation}
\mathcal M^{k+1}_{k,n-2}
\ :=\
\{F\in \mathcal D_{k,n-2}\,:\, \min F>|F|+k-1\}.
\end{equation}

\begin{prop}\label{prop:D-partitions}
The collections
\begin{equation}
\mathcal D^0_{k,n},\quad
\mathcal D^1_{k,n},\quad\ldots,\quad\mathcal D^k_{k,n},\quad
\mathcal D^{k+1}_{k,n},\quad \ldots,\quad \mathcal D^{2k}_{k,n}
\end{equation}
form a partition of $\mathcal D_{k,n}$; that is,
\begin{equation}\label{eq:D-partition}
\mathcal D_{k,n}
\ =\
\mathcal D^0_{k,n}\sqcup
\bigsqcup_{i=1}^{k}\mathcal D^i_{k,n}
\sqcup
\bigsqcup_{i=1}^{k}\mathcal D^{k+i}_{k,n}.
\end{equation}
Moreover, the collections
\begin{equation}
\mathcal M^1_{k,n-2},\quad \ldots,\quad\mathcal M^k_{k,n-2},\quad\mathcal M^{k+1}_{k,n-2}
\end{equation}
form a partition of $\mathcal D_{k,n-2}$; that is,
\begin{equation}\label{eq:M-partition}
\mathcal D_{k,n-2}
\ =\
\mathcal M^1_{k,n-2}\sqcup\cdots\sqcup
\mathcal M^k_{k,n-2}\sqcup
\mathcal M^{k+1}_{k,n-2}.
\end{equation}
\end{prop}

\begin{proof}
The second statement is obvious from definitions. We prove \eqref{eq:D-partition}. 

By definition, each of the collections $\mathcal D^0_{k,n}$ and $(\mathcal D^{i}_{k,n})_{i=1}^{2k}$ is a subset of $\mathcal D_{k,n}$. Conversely, let $F\in \mathcal D_{k,n}$. If $F$ is nonmaximal Schreier, then $F\in \mathcal D^0_{k,n}$. Suppose $F$ is maximal Schreier. Since $n\in F$ and $F$ is $k$-neighbored, the element $n$ has a neighbor in $F\backslash\{n\}$ within distance at most $k$. Hence, there exists $i\in [1,k]$ such that $n-i\in F$. Choose the smallest such $i$. Then $n-i\in F$ and $[n-i+1,n-1]\cap F=\emptyset$. Finally, $n-i$ either has a neighbor below or does not. If the former, $F\in \mathcal D^i_{k,n}$; if the latter, $F\in \mathcal D^{k+i}_{k,n}$. Therefore, we have 
\begin{equation}\label{d_union}\mathcal D_{k,n}
\ =\
\mathcal D^0_{k,n}\cup
\bigcup_{i=1}^{2k}\mathcal D^i_{k,n}.\end{equation}

The union in \eqref{d_union} is disjoint because a set cannot be both maximal and nonmaximal Schreier. Furthermore, among maximal Schreier sets, the smallest $i$ is unique, and for each $i$, the integer $n-i$ either has a neighbor below or does not, but not both. 
\end{proof}

\subsection{Bijective maps}

We now define bijections between pieces of the above partition of $\mathcal{D}_{k,n}$ and earlier collections $\mathcal{D}_{k, n-i}$ and their partitions:

\begin{align}
&f_0:\mathcal D_{k,n-1}\to \mathcal D^0_{k,n},
\quad
F\mapsto F+1,\\
&f_i:\mathcal M^i_{k,n-2}\to \mathcal D^i_{k,n},
\quad
F\mapsto (F-(i-2))\cup\{n\}, \quad 1\le i\le k,\\
&f_{i}:\mathcal D_{k,n-i-3}\to \mathcal D^{i}_{k,n},
\quad
F\mapsto (F-(\min F-|F|-2))\cup\{n-i+k,n\},\quad k+1\le i\le 2k,\\
&\mbox{ and}\nonumber\\
&f_{2k+1}:\mathcal D_{k,n-k-2} \to \mathcal M^{k+1}_{k,n-2},
\quad
F\mapsto F+k.
\end{align}

Figure \ref{fig: k-neighbored} provides a visual description of our maps and their domains and ranges.

\begin{figure}[H]
    \centering
    \includegraphics[width=1\textwidth]{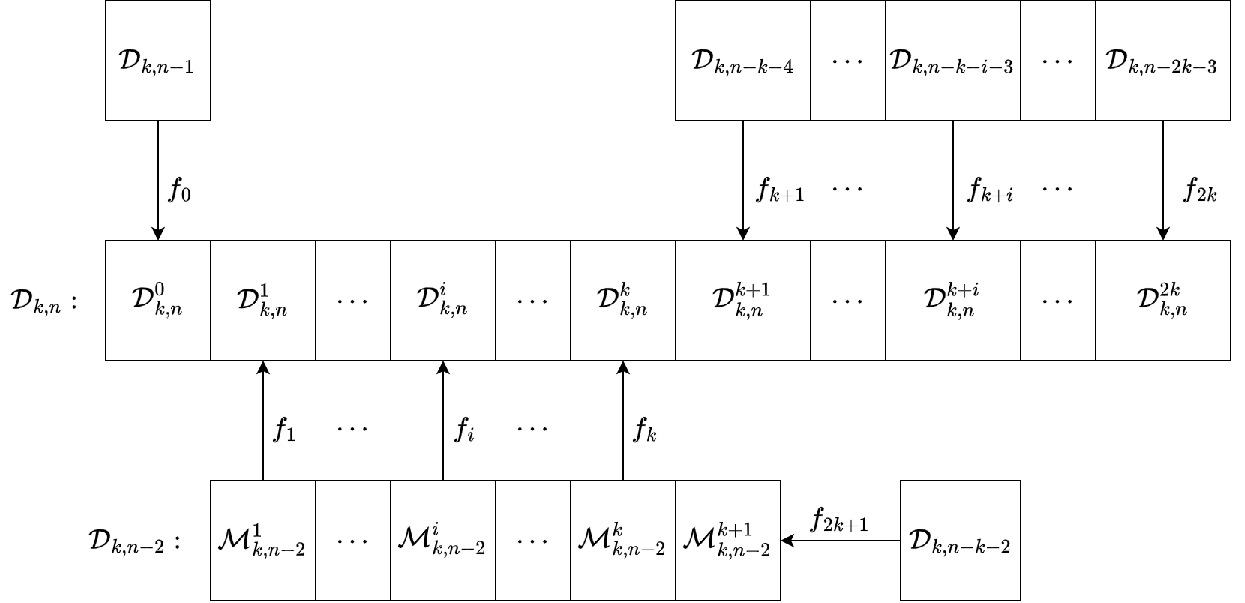}
    \caption{Visual summary of the bijections used in the proof of Theorem \ref{thm:k-neighbored-recurrence}. The diagram illustrates the domains and ranges of the bijections between a partition of $\mathcal{D}_{k,n}$ and the earlier collections $\mathcal{D}_{k,n-i}$.}
    \label{fig: k-neighbored}
\end{figure}
\begin{lem}\label{lem:f0-neighbored}
    The map $f_0$ is a bijection.
\end{lem}

\begin{proof}
    Let $F\in \mathcal{D}_{k,n-1}$. As $F$ is $k$-neighbored, $f_0(F) = F+1$ must also be $k$-neighbored. In addition, $\max F = n-1$ implies that $\max f_0(F) = n$. We have
    \begin{align}
        \min f_0(F) \ =\ \min F+1\ >\ |F|\ =\ |f_0(F)|,
    \end{align}
    so $f_0(F)$ is nonmaximal Schreier. Hence, $f_0$ is well-defined. 
    
    Injectivity of $f_0$ is obvious from the definition. To show surjectivity, pick $E \in \mathcal{D}_{k,n}^0$ and define $G := E-1$. As before, $G$ is $k$-neighbored because $E$ is $k$-neighbored. Also, $\max G = n-1$ because $\max E = n$. Lastly, the set $G$ is Schreier since
    \begin{align}
    \min G\ =\ \min E-1\ \ge\ |E|\ =\ |G|.
    \end{align}
    We have verified that $G \in \mathcal{D}_{k,n-1}$, so $f_0$ is surjective.
\end{proof}

\begin{lem}\label{lem:fi-neighbored}
    The maps $f_i$, with $1\le i\le k$, are bijective.
\end{lem}

\begin{proof}
    Fix $i\in [1,k]$ and a set $F\in \mathcal{M}_{k,n-2}^i$. As $F$ is $k$-neighbored and $n-2\in F$, the set $f_i(F)$ contains $n-i$, so $n$ has a neighbor. It follows that $f_i(F)$ is $k$-neighbored.
    Furthermore, $\max F = n-2$ implies that
    \begin{equation}[n-i+1, n-1]\cap f_i(F) \ =\ \emptyset,\end{equation}
    and $n-i$ has a lower neighbor in $f_i(F)$. By construction, $n \in f_i(F)$. Lastly, 
    \begin{align}
        \min f_i(F)\ =\ \min F-(i-2)\ =\ (|F|+i-1)-i+2\ =\ |F|+1\ =\ |f_i(F)|,
    \end{align}
    so $F$ is maximal Schreier. Hence, $f_i$ is well-defined. 
    
    Injectivity is obvious from the definition. To show surjectivity, pick $E \in \mathcal{D}_{k,n}^i$ and let $G := E\setminus \{n\}+i-2$. The set $G$ is $k$-neighbored because $n-i$ has a neighbor below, which implies that $E\backslash \{n\}$ is $k$-neighbored. Also, 
    \begin{equation}\max G \ =\ \max E\backslash\{n\} + i - 2\ =\ n-i+i-2\ =\ n-2.\end{equation}
    Finally, 
    \begin{equation}
    \min G\ =\ \min E+i-2\ =\ |E|+i-2\ =\ |G|+i-1.
    \end{equation}
    Hence, $G \in  \mathcal{M}^i_{k,n-2}$, so $f_i$ is surjective as $f_i(G)=E$.
\end{proof}

\begin{lem}\label{lem:fki-neighbored}
    The maps $f_i$, with $k+1\le i\le 2k$, are bijective.
\end{lem}

\begin{proof}
     Fix $i\in [k+1,2k]$ and a set $F\in \mathcal{D}_{k,n-i-3}$. Since $F$ is $k$-neighbored, $f_i(F)$ is also $k$-neighbored. By construction, $n$ and $n-i+k$ are in $f_i(F)$, and 
     \begin{equation}\max(F-(\min F-|F|-2))\ =\ n-i-3 - (\min F-|F|)+2\ \le \ n-i-1,\end{equation}
     so $[n-i+k+1,n-1] \cap f_i(F)=\emptyset$, and $n-i+k$ does not have a neighbor below in $f_i(F)$. Lastly, $F$ is maximal Schreier because
    \begin{align}
        \min f_i(F)\ =\ |F|+2\ =\ |f_i(F)|.
    \end{align}
    Hence, $f_i$ is well-defined. 
    
    We prove injectivity of $f_i$. Suppose that $f(A) = f(B)$ for $A, B\in \mathcal{D}_{k,n-i-3}$, i.e.,
    \begin{equation}(A-(\min A-|A|-2))\cup\{n-i+k,n\}\ =\ (B-(\min B-|B|-2))\cup\{n-i+k,n\}.\end{equation}
    Hence, $A - \min A = B- \min B$, which, together with $\max A = \max B = n-i-3$, implies that $\min A = \min B$. Therefore, the two sets $A$ and $B$ are equal.

    To show surjectivity, pick $E \in \mathcal{D}_{k,n}^i$. Let $G := E\setminus \{n-i+k,n\}$ and $F:=G-\max G+n-i-3$. Note that $G$ is not empty because $n\ge k+3$, so 
    \begin{equation}|E| \ =\ \min E \ \ge\ n-i+k\ \ge\ (k+3)-2k+k \ =\ 3.\end{equation}
    Since $E$ is $k$-neighbored, and $n-i+k$ has no neighbor below, the set $F$ is also $k$-neighbored. Note that 
    \begin{equation}\max F \ =\ \max G-\max G+n-i-3 \ =\ n-i+3,\end{equation}
    and
    \begin{align}
    \min F&\ =\ \min E-\max G+n-i-3\\
    &\ =\ |E|-\max G +n-i-3\\ 
    &\ \ge\  |E|-(n-i-1)+n-i-3\\
    &\ =\ |E|-2\ =\ |F|.
    \end{align}
    Hence, $F \in \mathcal{D}_{k,n-i-3}$. It remains to show that $f_i(F)=E$. We have
    \begin{align}
        f_i(F) &\ =\  (F-(\min F-|F|-2))\cup\{n-i+k,n\}\\
        &\ =\ (G-\min G+|G|+2)\cup \{n-i+k,n\}\\
        &\ =\  (G-\min E+|E|)\cup \{n-i+k,n\}\\
        &\ =\  (E\setminus \{n-i+k,n\})\cup \{n-i+k,n\}\ =\ E.
    \end{align}

    We have shown that $f_i$ is bijective.
\end{proof}

\begin{lem}\label{lem:f2k1-neighbored}
    The map $f_{2k+1}$ is a bijection.
\end{lem}

\begin{proof}
    Let $F\in \mathcal{D}_{k,n-k-2}$. Since $F$ is $k$-neighbored, $f_{2k+1}(F)$ is also $k$-neighbored. Since $\max F = n-k-2$, we have $\max f_{2k+1}(F) = n-2$. Lastly, 
    \begin{align}
        \min f_{2k+1}(F)\ =\ \min F+k\ \geq\ |F|+k\ >\ |f_{2k+1}(F)|+k-1.
    \end{align}
    Hence, the map $f_{2k+1}$ is well-defined.

    Injectivity is obvious from the definition. To show surjectivity, pick $E\in \mathcal{M}_{k,n-2}^{k+1}$ and define $G := E-k$. Since $E$ is $k$-neighbored, $G$ is $k$-neighbored. That $\max E = n-2$ gives $\max G = n-k-2$. Lastly, 
    \begin{align}
        \min G\ =\ \min E-k\ \geq\ |E|\ =\ |G|.
    \end{align}
    Hence, the set $G$ is in $\mathcal{D}_{k,n-k-2}$, and $f_{2k+1}$ is surjective because $f_{2k+1}(G)=E$.
\end{proof}

\subsection{Proof of Theorem \ref{thm:k-neighbored-recurrence}}

We now use the above bijections to prove Theorem~\ref{thm:k-neighbored-recurrence}.

\begin{proof}[Proof of Theorem \ref{thm:k-neighbored-recurrence}]
Suppose that $n\le k+2$. Let $F\in \mathcal{D}_{k,n}$. We have $|F|\ge 2$ because $F = \{n\}$ is not $k$-neighbored and thus, $\mathcal{D}_{k,n}\subset \mathcal{S}_{n}\backslash \{n\}$. Conversely, if $F\in\mathcal{S}_{n}\backslash \{n\}$, then $|F|\ge 2$ implies that $\min F\ge 2$. Hence,
\begin{equation}\max F - \min F\ \le\ n - 2\ \le\ (k+2)-2 \ =\ k,\end{equation}
so $F$ is $k$-neighbored. Due to \eqref{bfs}, we have
\begin{equation}|\mathcal{D}_{k,n}|\ =\ |\mathcal{S}_{n}\backslash \{n\}|\ =\ F_n - 1.\end{equation}

We prove the second statement of Theorem \ref{thm:k-neighbored-recurrence}.
    By Proposition~\ref{prop:D-partitions},
    \begin{equation}\label{eq:D-count-decomposition}
    |\mathcal D_{k,n}|
    \ =\
    |\mathcal D^0_{k,n}|
    +
    \sum_{i=1}^{k}|\mathcal D^i_{k,n}|
    +
    \sum_{i=k+1}^{2k}|\mathcal D^{i}_{k,n}|.
    \end{equation}
    Lemma~\ref{lem:f0-neighbored} gives a bijection between $\mathcal D_{k,n-1}$  and $\mathcal D^0_{k,n}$, so
    \begin{equation}\label{eq:D0-count}
    |\mathcal D^0_{k,n}|\ =\ |\mathcal D_{k,n-1}|.
    \end{equation}
    Lemma~\ref{lem:fi-neighbored} gives bijections between $\mathcal M^i_{k,n-2}$ and $\mathcal D^i_{k,n}$ for $1\le i\le k$. Hence, by Proposition~\ref{prop:D-partitions},
    \begin{equation}
    \sum_{i=1}^{k}|\mathcal D^i_{k,n}|\ =\ \sum_{i=1}^{k}|\mathcal M^i_{k,n-2}| \ =\ |\mathcal D_{k,n-2}|-|\mathcal M^{k+1}_{k,n-2}|.
    \label{eq:Di-count}
    \end{equation}
    Lemma~\ref{lem:f2k1-neighbored} gives the bijection between $\mathcal D_{k,n-k-2}$ and $\mathcal M^{k+1}_{k,n-2}$. Therefore
    \begin{equation}\label{eq:M-count}
    |\mathcal M^{k+1}_{k,n-2}|\ =\ |\mathcal D_{k,n-k-2}|.
    \end{equation}
    Combining \eqref{eq:Di-count} and \eqref{eq:M-count}, we get
    \begin{equation}\label{eq:Di-final-count}
    \sum_{i=1}^{k}|\mathcal D^i_{k,n}|\ =\ |\mathcal D_{k,n-2}|-|\mathcal D_{k,n-k-2}|.
    \end{equation}
    Finally, Lemma~\ref{lem:fki-neighbored} gives bijections between $\mathcal D_{k,n-i-3}$ and $\mathcal D^{i}_{k,n}$ for $k+1\le i\le 2k$. Thus,
    \begin{equation}
    \sum_{i=k+1}^{2k}|\mathcal D^{i}_{k,n}|\ =\ \sum_{i=k+1}^{2k}|\mathcal D_{k,n-i-3}|\ =\ \sum_{i=k+4}^{2k+3}|\mathcal D_{k,n-i}|.
    \label{eq:Dki-count}
    \end{equation}
    Substituting \eqref{eq:D0-count}, \eqref{eq:Di-final-count}, and \eqref{eq:Dki-count} into \eqref{eq:D-count-decomposition} gives
    \begin{equation}
    |\mathcal D_{k,n}|\ =\ |\mathcal D_{k,n-1}|+|\mathcal D_{k,n-2}|-|\mathcal D_{k,n-k-2}|+
    \sum_{i=k+4}^{2k+3}|\mathcal D_{k,n-i}|,
    \end{equation}
   as desired. 
\end{proof}


\section{Recurrences from counting $k$-isolated Schreier sets}\label{sect_k_isolated}

We investigate Schreier sets whose elements are more than $k$ apart. 
To prove the recurrence in Theorem \ref{thm:k-isolated-recurrence}, we decompose $\mathcal{I}_{k,n}$ into two subcollections of maximal and nonmaximal Schreier sets. Let
\begin{align}
    &\mathcal{I}_{k,n}^1\ :=\ \{F\subset[n]:n\in F, F \mbox{ is nonmaximal Schreier, and } F \mbox{ is } k \mbox{-isolated}\}\mbox{ and}\\
    &\mathcal{I}_{k,n}^2\ :=\ \{F\subset[n]:n\in F, F \mbox{ is maximal Schreier, and } F \mbox{ is } k \mbox{-isolated}\}.
\end{align}
For $n\ge k+3$, define
\begin{align}
    &f_1:\mathcal{I}_{k,n-1}\to\mathcal{I}_{k,n}^1,\quad F\mapsto F+1\mbox{ and}\\
    &f_2:\mathcal{I}_{k,n-k-2}\to\mathcal{I}_{k,n}^2,\quad F\mapsto \{|F|+1\}\cup (F+k+2).
\end{align}
It suffices to prove that these maps are bijective.

    \begin{lem}\label{lem:f1-isolated}
        The map $f_1$ is a bijection.
    \end{lem}
    
    \begin{proof}
        For $F\in\mathcal{I}_{k,n-1}$, since $\max F = n-1$, we have $\max f_1(F) = n$. Note that
        \begin{equation}\min f_1(F)\ =\ \min F +1\  \ge\  |F| +1\ >\ |F|\ =\ |f_1(F)|.\end{equation}
        Hence, the set $f_1(F)$ is nonmaximal Schreier. Since $F$ is $k$-isolated, we know that $f_1(F) = F+1$ is also $k$-isolated. Therefore, the map $f_1$ is well-defined.

        Injectivity of $f_1$ is obvious from the definition. To show surjectivity, let $E\in\mathcal{I}_{k,n}^1$ and let $F:=E-1$. Since $E\subset[n]$ and $n\in E$, we have $F\subset[n-1]$ and $n-1\in F$. Note that $E$ is nonmaximal Schreier,  so $\min E\ge |E|+1$, and
        \begin{equation}\min F \ =\ \min E -1\ \ge\ |E|\ =\ |F|.\end{equation}
        Therefore, $F$ is a Schreier set. Moreover, since $E$ is $k$-isolated, $F$ is also $k$-isolated. Hence, $F\in \mathcal{I}_{k,n-1}$, and $f_1(F) = E$ gives that $f_1$ is surjective.
    \end{proof}
    
    \begin{lem}\label{lem:f2-isolated}
        The map $f_2$ is a bijection.
    \end{lem}
    
    \begin{proof}
        For each $F\in\mathcal{I}_{k,n-k-2}$, since $\max F = n-k-2$, we have $\max f_2(F) = n$. Note that
        \begin{equation}\min (F+k+2) - (|F| + 1)\ \ge\  k+1.\end{equation}
        Hence, $f_2(F)$ is $k$-isolated, and
        \begin{equation}\min f_2(F)\  =\  |F|+1\ =\ |f_2(F)|,\end{equation}
        so $f_2(F)$ is maximal Schreier. 
        We have shown that $f_2(F)\in\mathcal{I}_{k,n}^2$. Therefore, $f_2$ is well-defined.

        Injectivity of $f_2$ is obvious from the definition. To show the surjectivity of $f_2$, let $E\in \mathcal{I}_{k,n}^2$ and let $F:=(E\backslash\{\min E\})-k-2$. Since $E$ is maximal Schreier, we have 
        \begin{align}
            f_2(F) &\ =\ (F+k+2)\cup \{|F|+1\}\\
            &\ =\ (E\backslash\{\min E\})\cup \{|E|\}\ =\ (E\backslash\{\min E\})\cup \{\min E\}\ =\ E.
        \end{align}
        Note that $\max F = n-k-2$ because $\max E = n$ and $E\neq \{n\}$, while 
        \begin{align}
            \min F &\ =\ \min (E\backslash\{\min E\})-k-2\\
            &\ \ge\ (\min E + k+1)-k-2\ =\ |E|-1\ =\ |F|.
        \end{align}
        implies the set $F$ is Schreier. Furthermore, $F$ is $k$-isolated because $E$ is $k$-isolated. Therefore, $F\in\mathcal{I}_{k,n-k-2}$. We have thus verified that $f_2$ is surjective.
    \end{proof}

 \begin{proof}[Proof of Theorem \ref{thm:k-isolated-recurrence}]
 We first compute the initial terms  $(|\mathcal{I}_{k,n}|)_{n=1}^{k+2}$. Let $1\le n\le k+2$. Clearly, the singleton $\{n\}$ is in $\mathcal{I}_{k,n}$. We claim that there is no $F\in\mathcal{I}_{k,n}$ with $|F|\ge 2$. Indeed, since $F$ is Schreier,
\begin{equation*}
\min F\ \ge\ |F|\ \ge\ 2.
\end{equation*}
Hence, 
\begin{equation*}
n-\min F\ \le\ n-2\ \le\ k,
\end{equation*}
which contradicts that $F$ is $k$-isolated. Therefore, we have
$|\mathcal{I}_{k,n}| = 1$, if $1\le n\le k+2$. For $n\geq k+3$, by
 Lemmas \ref{lem:f1-isolated} and \ref{lem:f2-isolated}, we have
 \begin{equation}
 |\mathcal{I}_{k,n}|\ =\ |\mathcal{I}^1_{k,n}| + |\mathcal{I}^2_{k,n}|\ =\ |\mathcal{I}_{k, n-1}| + |\mathcal{I}_{k, n-k-2}|.
 \end{equation}
 \end{proof}   

\section{Schreier sets that are closed under $2$-averages}\label{sect_2_averages}

To prove Theorem \ref{mtsigma}, we employ two interesting properties of sets in $\mathcal{G}_{2n+1}$. If $F$ is in $\mathcal{G}_{2n+1}$, then $\min F - |F|\neq 1$; furthermore, $F$ must be an arithmetic progression. 

\begin{lem}\label{lem:evenodd}
For all $F \in \mathcal{G}_{2n+1}$, we have $\min F \neq |F| + 1$. Furthermore, 
there exists an odd $q$ such that
\begin{equation}\max_i F \ =\ \max_{i+1} F + q\mbox{ for all }i\in [1, |F| - 1],\end{equation}
 where $\max_i F$ denotes the $i$th largest element of the set $F\subset \mathbb{N}$.
\end{lem}
\begin{proof}
We claim that $\max_2 F = 2k$ for some $k\in [1, n]$; otherwise, if $\max_2  F = 2k+1$ for some $k\in [1,n-1]$, then
\begin{equation}\frac{\max F + (2k+1)}{2}\ =\ \frac{(2k+1)+(2n+1)}{2}\ =\ k+n + 1\mbox{ must be in }F;\end{equation}
however,
\begin{equation}\max_2 F \ =\ 2k+1 \ <\ k+n+1 \ <\ 2n+1 \ =\ \max F,\end{equation}
which contradicts that $k+n+1$ is in $F$.
In general, we have 
\begin{equation}\max_i F\mbox{ is }\begin{cases}\mbox{ odd},&\mbox{ if } 2\nmid i,\\ \mbox{ even}, &\mbox{ if }2| i.\end{cases}\end{equation}
It follows that $\min F$ is odd if and only if $|F|$ is odd. Therefore, $\min F\neq |F| + 1$.

We prove the second part of the lemma. Let $i\in [2, |F|-1]$. 
From above, $\max_{i-1} F$ and $\max_{i+1} F$ have the same parity, which implies that 
$(\max_{i-1} F + \max_{i+1} F)/2$ is in $F$. Hence, 
\begin{equation}\max_i F\ =\ \frac{\max_{i-1} F + \max_{i+1} F}{2},\end{equation}
so $\max_{i} F - \max_{i+1} F = \max_{i-1} F - \max_{i} F$. Therefore, there exists an odd $q$ with
\begin{equation}\max_{i} F - \max_{i+1} F\ =\ q,\mbox{ for all }i\in [1, |F|-1].\end{equation}
\end{proof}

Lemma \ref{lem:evenodd} guides us to partition $\mathcal{G}_{2n+1}$, with $n\ge 0$, into
\begin{align*}
\mathcal{G}_{2n+1}^1&\ =\ \{ F \in \mathcal{G}_{2n+1}\,:\, \min F \geq |F| + 2\}, \mbox{ and }\\
\mathcal{G}_{2n+1}^2&\ =\ \{ F \in \mathcal{G}_{2n+1}\,:\, F \text{ is maximal Schreier}\}.
\end{align*}
For $n \ge 1$, define
\begin{align*}
&f_1: \mathcal{G}_{2n-1} \rightarrow \mathcal{G}_{2n+1}^1,\quad  F \mapsto F+2,\mbox{ and }\\
&f_2: \{ a\in\mathbb{N}: a \text{ divides } n\} \rightarrow \mathcal{G}_{2n+1}^2, \quad  a \mapsto \left\{a+1 + i\cdot \left(\frac{2n-a}{a}\right)\,:\, 0\le i\le a\right\}.
\end{align*}

\begin{lem}\label{lem_div2_1}
The map $f_1$ is a bijection.
\end{lem}
\begin{proof}
Let $F \in \mathcal{G}_{2n-1}$. By the definition of $f_1$, we have $\max f_1(F)=2n+1$. Fix $a, b \in f_1(F)$ with 
$(a+b)/2\in \mathbb{Z}$. Then $a-2$ and $b-2$ are in $F$, and $((a-2) + (b-2))/2$ is in $\mathbb{Z}$ and thus, is in $F$. 
Hence, 
\begin{equation}\frac{a+b}{2} \ =\ \frac{(a-2) + (b-2)}{2} + 2\mbox{ is in }f_1(F).\end{equation}
Moreover, 
\begin{equation}
\min f_1(F) \ =\ \min F + 2 \ \ge\ |F|+2 \ =\ |f_1(F)| + 2.
\end{equation}
Hence, $f_1$ is well-defined.

Injectivity is obvious from the definition. To show surjectivity, pick $E \in \mathcal{G}_{2n+1}^1$. Let $G:= E - 2$. Then $\max G = 2n-1$. Fix $a,b \in G$ with $(a+b)/2\in \mathbb{Z}$. Then $(a+2)$ and $(b+2)$ are in $E$, and $((a+2)+(b+2))/2$ is an integer and thus, is in $E$. Hence,
\begin{equation}\frac{a+b}{2}\ =\ \frac{(a+2)+(b+2)}{2}-2\mbox{ is in }G.\end{equation}
Lastly, we have
\begin{align}
\min G \ = \ \min E - 2 \ \ge\ |E| \ =\ |G|.
\end{align}
Hence, $G \in \mathcal{G}_{2n-1}$, and $f_1$ is surjective.
\end{proof}

\begin{lem}\label{lem_div2_2}
The map $f_2$ is a bijection.
\end{lem}
\begin{proof}
Pick $r\in  \{a: a \text{ divides }n\}$. We have $\max f_2(r) = 2n+1$ by the definition of $f_2$. Next, fix $i, j \in [0, r]$. If \begin{equation}\frac{r+1+i\frac{2n-r}{r}+r+1+j\frac{2n-r}{r}}{2} \ =\ r+1 + (i+j)\frac{n}{r} - \frac{i+j}{2}\end{equation}
is an integer,
then $i+j$ is even, and 
\begin{equation}\frac{r+1+i\frac{2n-r}{r}+r+1+j\frac{2n-r}{r}}{2} \ =\ r+1+\frac{i+j}{2}\cdot \frac{2n-r}{r}\mbox{ is in } f_2(r).\end{equation}
Lastly, the set $f_2(r)$ is maximal Schreier because
\begin{equation}
\min f_2(r)\ =\ r+1\ =\ |f_2(r)|.
\end{equation}
We have verified that $f_2$ is well-defined.

Injectivity is obvious from the definition. To show surjectivity, pick $E\in \mathcal{G}_{2n+1}^2$. Let $g := \min E -1$ and $q := \max E-\max_2E$, which is odd. By Lemma~\ref{lem:evenodd}, 
\begin{align}
    \min E&\ =\ 2n+1 - (|E|-1)q\\
    &\ =\ 2n + 1 - (\min E - 1)q,
\end{align}
so 
\begin{equation}n \ =\ (\min E - 1)\frac{q+1}{2}\ =\ g\cdot \frac{q+1}{2}.\end{equation}
Hence, $g \in \{ a: a \text{ divides } n\}$, and 
\begin{align}
    f_2(g)&\ =\ \left\{g+1, g+1 + \frac{2n-g}{g}, \ldots, g+1+g\frac{2n-g}{g}\right\}\\
    &\ =\ \left\{\min E, \min E+q, \ldots, \min E + g\cdot q\right\}\ =\ E.
\end{align}
Therefore, the map $f_2$ is surjective.
\end{proof}

\begin{proof}[Proof of Theorem \ref{mtsigma}]
    Theorem \ref{mtsigma} follows immediately from Lemmas \ref{lem_div2_1} and \ref{lem_div2_2}.
\end{proof}


\section{Future work}
Beyond the neighborhood conditions on standard Schreier sets considered here, we expect that many other interesting variants remain to be explored. One natural extension is to count the cardinality for similar neighborhood restrictions under the generalized Schreier condition
$$q\min F\ \ge\ p|F|,$$
or to require $F$ to contain elements from a special sequence rather than from $[n]$. Furthermore, we mention two notable problems from the literature. The first one asks for a recurrence for the collections of multisets
$$\mathcal{A}^{(s)}_{p,q,n}\ :=\ \{F\subset\{\underbrace{1,\ldots, 1}_s, \ldots, \underbrace{n-1,\ldots, n-1}_s, n\}\,:\, n\in F\mbox{ and }q\min F\ge p|F|\}.$$
Note that \cite[Theorem 1.1]{BCF} and \cite[Theorem 1.2]{CGKMTV} solved the cases $s = 1$ and $q = 1$, respectively. 
The second one investigates the nonlinear Schreier condition $(\min F)^q\ge |F|^p$ for $p, q\in \mathbb{N}$. The case $q = 1$ was already addressed in \cite[Theorem 4]{CIMSZ}.

\end{document}